\documentclass[10pt]{amsart}
\usepackage[latin1]{inputenc}
\usepackage[T1]{fontenc}
\usepackage{amsfonts}
\usepackage{amsmath}
\usepackage{amssymb}
\usepackage{graphicx}
\usepackage{pstricks}
\usepackage{pst-grad}
\usepackage{multido}
\usepackage{amsthm}
\usepackage{geometry}
\usepackage[latin1]{inputenc}
  \usepackage[T1]{fontenc}
  \usepackage{amsmath,amssymb}
  \usepackage{showlabels}
  \usepackage[all]{xy}

\usepackage[absolute]{textpos}
\definecolor{purple}{rgb}{0.8,0.12,0.8}
\definecolor{orange}{rgb}{1.0,0.7,0.0}
\definecolor{pink}{rgb}{1,0.5,0.8}
\definecolor{blackg}{rgb}{0.1,0.25,0.1}
\definecolor{ForestGreen}{cmyk}{0.91,0,0.88,0.42}
\definecolor{Turquoise}{cmyk}{0.85,0,0.20,0}

\newcommand{\blambda}{\boldsymbol{\lambda}}

\newcommand{\bmu}{\boldsymbol{\mu}}
\newcommand{\bemptyset}{\boldsymbol{\emptyset}}

\newtheorem{Th}{Theorem}[section]

\newtheorem{Cor}[Th]{Corollary}
\newtheorem{Prop}[Th]{Proposition}
\newtheorem{Def-Prop}[Th]{Definition-Proposition}

\theoremstyle{definition}
\newtheorem{Def}[Th]{Definition}
\newtheorem{Exa}[Th]{Example}

\theoremstyle{remark}
\newtheorem{Rem}[Th]{Remark}

\begin{document}

\title{\mbox{Constructible representations and basic sets in type $B$}}
\author{Nicolas Jacon}
\begin{abstract}
We  study the parametrizations of simple modules provided by the theory of basic sets 
 for all finite Weyl groups. In the case of type $B_n$, we  show the existence of basic sets
  for the matrices of constructible representations. Then we study  bijections between the various basic sets 
  and show that they are controlled by the  matrices of the constructible representations.

\end{abstract}

\maketitle
\section{Introduction}

  One of the main problems in the modular representation theory of Hecke algebras of finite Weyl groups is to find ``good" 
   parametrizations of the set of simple modules.
   There is a natural way to solve that problem by studying the associated decomposition matrices. 
   In fact, 
    in characteristic zero,  using these matrices, 
     it is possible to prove the existence of certain indexing sets called
      ``basic sets"   which are in natural bijection with  the set of simple modules for the Hecke algebra.  
      Once we have the existence of these sets, it is another problem 
      to have an explicit characterization of them. This has been achieved in type $A_{n-1}$ by Dipper and James \cite{DJ0}, in type $B_n$ combining works by M. Chlouveraki, M.Geck, and the author  \cite{CJ},  \cite{GJBn}, in type $D_n$ by M. Geck \cite{GeckD} and the author \cite{JDn} 
  and for the exceptional types by Geck,  M\"uller and Lux, \cite{geckexcep},  \cite{gecklux}, \cite{muller}, \cite{muller2}. Importantly, these results remain valid in positive characteristic under the 
        assumptions of certain Lusztig conjectures. We refer to \cite{Gsurvey} for a survey of these results.

  In the case of type $B_n$, the theory of basic sets provides several natural ways to label the same set of simple modules. In this paper,
   we are mainly  interested on the connections between these various basic sets. 
  First, we show the existence of analogues of  basic sets for other types of representation introduced by Lusztig: the constructible representations (see \cite{lu}). 
           In fact, given the matrix of the constructible characters for a choice of parameters (this includes the case where these parameters are negative), 
     we show the existence of two different associated basic sets. As a consequence, we obtain two natural ways to parametrize
       the constructible characters, extending the work of Lusztig \cite{lu1}, \cite{lu2}, \cite{lu}.  In addition, we describe the bijection between these two basic sets. 
       All these results use as a crucial tool the works of Lusztig, Leclerc and Miyachi \cite{LM} and the combinatorics developed therein.

       The last part of the paper is devoted to the study of the various bijections between the basic sets in type $B_n$. 
       It turns out that
        the  bipartitions labelling these sets are difficult to describe in general (we only have in principle  a recursive description of them).
          Then, in the same spirit as in \cite{JL}, we show the existence of an action of the affine extended symmetric group $\widehat{\mathfrak{S}}_2$ on these basic sets. 
           We observe the two following remarkable facts:
       \begin{itemize}
       \item there is an easy description of the basic sets lying in  a fundamental domain associated with the above action, 
       \item  the action of $\widehat{\mathfrak{S}}_2$ on the set of basic sets can be explicitly described combining our previous results with results obtained in \cite{JL}.
       \end{itemize}
    It is then possible to describe the basic sets as orbits of the elements of the fundamental domain under this action. 
       Finally,  we remark  that this action is in some sense controlled 
    by the matrices of constructible representations. In particular the bijections between the various basic sets can be essentially read through these matrices. 
       We end the paper with an explanation of this phenomenon. \\

\section{Decomposition matrices for Hecke algebras}
Let  $(W,S)$ be a finite Weyl group. We  assume that we have a decomposition $S=S_+ \sqcup S_-$ where no elements of $S_+$ is conjugate to an element
 of $S_-$.  Let $\phi:S \to \mathbb{Z}$ be such that 
 $$\phi (s)=\phi (s')\text{ if }(s,s')\in S^2_{+}\text{ and }\phi (s)=\phi (s')\text{ if }(s,s')\in S^2_{-} \ \ (\star)$$
  Let $q$ 
  be an indeterminate and choose $q^{1/2}$ a root of $q$.  
 We then  have an associated 
  Iwahori-Hecke algebra $\mathcal{H}(W,S,\phi)$ over $A=\mathbb{Z}[q^{1/2},q^{-1/2}]$. The basis is given by $\{T_w\}_{w\in W}$ and the multiplication is determined by the following
   rules:
   $$\left\{\begin{array}{ll}
   T_w T_{w'}=T_{ww'} & \textrm{if }l(ww')=l(w)+l(w')\\
   (T_s -q^{\phi(s)})(T_s+1)=0
   \end{array}\right.$$
In this section, we study the representation theory of these algebras in both the semisimple and the modular case and give extensions of some definitions and properties  
 which were previously  only known when $\phi (S)\subset \mathbb{N}$

\subsection{Decomposition matrices}\label{int}   Let $K$ be the field of fractions of $A$. Then by \cite[\S 9.3.5]{GP}, 
  the algebra $\mathcal{H}_K(W,S,\phi):=K\otimes_A \mathcal{H}(W,S,\phi)$
   is split semisimple and by Tits' deformation theorem, we have a canonical bijection between $\operatorname{Irr}( \mathcal{H}_K(W,S,\phi))$ 
    and $\operatorname{Irr}(W)$. Let $\Lambda$ is an indexing set for $\operatorname{Irr}(W)$:
    $$\operatorname{Irr}( W)=\left\{ E^{\blambda}\ |\ \blambda\in \Lambda\right\}.$$
    We then have:
    $$\operatorname{Irr}( \mathcal{H}_K(W,S,\phi))=\left\{ V_{\phi}^{\blambda}\ |\ \blambda\in \Lambda\right\}$$

Let $k$ be a field and $\xi\in k^{\times}$ be an element which has
a square root in $k^{\times}$. Then there is a ring homomorphism $\theta:A\to k$ such
that $\theta (q )=\xi$. Considering $k$ as an $A$-module via $\theta$, we set
$\mathcal{H}_k(W,S,\phi):=k\otimes_{A}\mathcal{H}_A(W,S,\phi)$.  As noted above, we 
 have a canonical way to parametrize the simple modules for $\mathcal{H}_K(W,S,\phi)$. It is also desirable 
  to obtain a ``good" parametrization of the simple $\mathcal{H}_k(W,S,\phi)$-modules. As $\mathcal{H}_k(W,S,\phi)$ is 
   not semisimple in general, Tits' deformation theorem cannot be applied. However, following \cite[\S 4.10]{Gsurvey}, one can use the associated decomposition
    matrix to solve that problem.  Let  $\blambda\in \Lambda$ and let 
 $$\begin{array}{cccl}
  \rho^{\blambda}  \colon &\mathcal{H}_K(W,S,\phi) &\rightarrow &M_d(K)\\
    & T_w & \mapsto & (a_{ij} (T_w))_{1\leq i,j\leq d}
    \end{array}$$
   be a matrix representation 
  affording the module $V_{\phi}^{\blambda}\in \operatorname{Irr}( \mathcal{H}_K(W,S,\phi))$ of dimension $d$. 
   The ideal  $\mathfrak{p}=\ker(\theta)$ is a prime ideal in $A$ and the localization 
$A_{\mathfrak{p}}$ is a regular local ring of Krull dimension $\leq 2$. Hence, by Du--Parshall--Scott 
\cite[\S 1.1.1]{DPS1}, we can assume that $\rho^{\blambda}$ satisfies the condition
\[ \rho^{\blambda}(T_w) \in M_d  (A_{\mathfrak{p}}) \qquad \mbox{for all $w\in W$}.\]
Now, $\theta$  extends to a ring homomorphism $\theta_{\mathfrak{p}}
\colon A_{\mathfrak{p}} \rightarrow k$. Applying $\theta_{\mathfrak{p}}$, we obtain a 
representation
 $$\begin{array}{llll} 
 \rho^{\blambda}_{k,\xi}:& \mathcal{H}_k(W,S,\phi)&\rightarrow& M_d  (k), \\
&T_w &\mapsto &(\theta_{\mathfrak{p}}(a_{ij}(T_w)))_{1\leq i,j\leq d}.
\end{array}$$
This representation may no longer be irreducible. For any $M \in
\operatorname{Irr}(      \mathcal{H}_k(W,S,\phi)  )$, let $[   V_{\phi}^{\blambda}:M]$ be the multiplicity of $M$ as a
composition factor of the $ \mathcal{H}_k(W,S,\phi)$-module affording $\rho^{\blambda}_{k,\xi}$.
This is well defined by  \cite[\S 1.1.2]{DPS1}. 
Thus, we obtain a well-defined matrix
\[ D^{\phi}_{\theta}=\left([ V_{\phi}^{\blambda} :M]\right)_{\blambda\in \Lambda,M \in \operatorname{Irr}( \mathcal{H}_k(W,S,\phi)) }\]
which is called the { decomposition matrix} associated with $\theta$. Let 
$R(\mathcal{H}_K(W,S,\phi))$ (resp. $R(\mathcal{H}_k(W,S,\phi))$) be the Grothendieck group
 of finitely generated $\mathcal{H}_K(W,S,\phi)$-modules (resp. $\mathcal{H}_k (W,S,\phi)$-modules). 
  It is generated by the classes $[U]$ of the simple  $\mathcal{H}_K(W,S,\phi)$-modules (resp. $\mathcal{H}_k (W,S,\phi)$-modules) $U$.
   Then we obtained a well defined  decomposition map:
$$d^{\phi}_{\theta}:R(\mathcal{H}_K(W,S,\phi))\to R(\mathcal{H}_k(W,S,\phi))$$
such that for all $\lambda\in \Lambda$  we have:
$$d^{\phi}_{\theta}([V^{\blambda}_{\phi}])=\sum_{M\in \operatorname{Irr}(\mathcal{H}_k(W,S,\phi))} [ V_{\phi}^{\blambda} :M] [M]$$

The notion of basic sets of simple modules for Hecke algebras has first been considered by Geck in \cite{Gcells}.
 The definition depends on the decomposition matrix and has been originally 
   given in the case where $\phi$ is constant and positive and then in the case where $\phi$ is positive in \cite{Gsurvey}.
    Using the above discussion, we will be able to generalize these notions to all possible $\phi$.

\subsection{Basic sets}\label{basic}

%\section{Basic sets and constructible characters}
In this section, we adopt the following notations. Let $\phi:S\to \mathbb{Z}$ be a map satisfying 
 $$\phi (s)=\phi (s')\text{ if }(s,s')\in S^2_{+}\text{ and }\phi (s)=\phi (s')\text{ if }(s,s')\in S^2_{-}  \qquad   (\star).$$
  We denote by 
  $|\phi|:S\to \mathbb{Z}$ the map such that $|\phi| (s)=|\phi (s)|$ for all $s\in S$. Note that
   this map satisfy $(\star)$. Set:
   $$S^{-}:=\left\{ s\in S\ |\ \phi (s)<0\right\}.$$
   Let $\varepsilon: W\to \mathbb{Q}^{\times}$ be the one dimensional representation of $W$ such that $\varepsilon (s)=1$
   if $s\in S^+$ and $\varepsilon (s)=-1$ if $s\in S^-$. For $\blambda\in \Lambda$, the module $(E^{\blambda})^{\varepsilon}$
    remains simple and  
   we define $\blambda^{\varepsilon}\in \Lambda$ 
    such that $E^{\blambda^{\varepsilon}}\simeq  (E^{\blambda})^{\varepsilon}$.

We turn to the definition of basic sets associated with a specialization of the Hecke algebra.

%By the general theory of Lusztig, we have a Kazhdan-Lusztig  basis for $\mathcal{H}_K(W,S,\phi)$. 
% Using this, we can associate to each 
 %simple $\mathcal{H}_K(W,S,\phi)$-module an $a$-invariant $a^{\phi}(E^{\blambda}):\in \mathbb{Z}$
\begin{Def}\label{basique}
%Let  be a map.
% Compatibility with the Left cell structure ?
We say that $\mathcal{H}(W,S,\phi)$  admits a basic set $\mathcal{B}(\phi)\subset \Lambda$ with respect to
  $\theta:A\to k$ and to a map $\alpha^{\phi}:\Lambda\to \mathbb{Q}$ if and only if:
  \begin{enumerate}
\item   For all $M\in\operatorname{Irr}( \mathcal{H}_k(W,S,\phi))$ there exists $\lambda_M\in \mathcal{B}(\phi)$ such that
 $$[V_{\phi}^{\blambda_M},M] =1\textrm{ and }\alpha^{\phi}(\bmu)>\alpha^{\phi}(\blambda_M)\textrm{ if }[V_{\phi}^{\bmu},M] \neq 0$$
\item The map 
 $$\begin{array}{cll}
 \operatorname{Irr}( \mathcal{H}_k(W,S,\phi))   &\to& \mathcal{B}(\phi)\\
M &\mapsto& \lambda_M
\end{array}$$
is a bijection \end{enumerate}
\end{Def}
Assume that $\mathcal{H}(W,S,\phi)$   admits a basic set $\mathcal{B}(\phi)\subset \Lambda$ with respect to
  $\theta$ and to a map $\alpha^{\phi}:\Lambda\to \mathbb{Q}$. This implies that the associated decomposition matrix has a lower 
  triangular shape with one along the diagonal for a ``good" order on  $\Lambda$ induced by the map $\alpha^{\phi}$. Hence, 
   it gives a way to label 
  $\operatorname{Irr}( \mathcal{H}_k(W,S,\phi))$.

 It is now natural to ask if these basic sets always exist.
The question 
 has been considered in \cite{GR}, \cite{Gcells}, \cite{Gsurvey} (see \cite{GJlivre} for 
 a complete survey on this theory)  and in  \cite{CJ} (where the question of existence 
 of basic sets in characteristic $0$ and for any weight function is complete), 

 The first step is to define the canonical map 
  $a^{\phi}:\Lambda\to \mathbb{N}$ which will play the role of $\alpha^{\phi}$. This can be  done using the
   symmetric algebra structure of  $\mathcal{H}(W,S,\phi)$ 
 We define a linear map $\tau \colon \mathcal{H}(W,S,\phi) \rightarrow A$ by 
\[ \tau(T_1)=1 \qquad \mbox{and} \qquad \tau(T_w)=0 \quad 
\mbox{for $w\neq 1$}.\]
Then one can show that $\tau$ is a trace function and we have 
\[ \tau(T_wT_{w'})=\left\{\begin{array}{cl} \pi(w) & \qquad 
\mbox{if $w'=w^{-1}$},\\ 0 & \qquad \mbox{otherwise}.\end{array}\right.\]
This implies that $\mathcal{H}(W,S,\phi)$  is a symmetric algebra (see \cite[Ch. 7]{GP} for a study of the representation theory  of this type of algebras). The above trace form
 extends to a trace form $\tau_K \colon \mathcal{H}_K (W,S,\phi) \rightarrow K$. Now since  
  $ \mathcal{H}_K (W,S,\phi)$ is split semisimple, we have
  $$\tau_{K} (T_w)=\sum_{\blambda\in \Lambda} \frac{1}{c_{\blambda}} \textrm{trace}(T_w,V^{\blambda}_{\phi}),\qquad \textrm{ for all }w\in W,$$
  where $c_{\blambda}\in A$ is called the Schur element associated to $\blambda\in \Lambda$.  For all  $\blambda\in \Lambda$, we now have 
  $$c_{\blambda}=f_{\blambda} q^{-a^{\phi}_{\blambda}}+\textrm{ combination of higher powers of }q,$$
  where  $f_{\blambda}$ and $a^{\phi}_{\blambda}$ are both integers such that $f_{\blambda}>0$ and $a^{\phi}_{\blambda}\geq 0$ 
  (see \cite[Ch. 20]{GP} for details.) The map 
  $$\begin{array}{cccc}
  a^{\phi}:&\Lambda&\to &\mathbb{N}\\
   & \blambda & \mapsto & a^{\phi}_{\blambda}
  \end{array}$$
  is called the Lusztig $a$-function.
  
  In the next theorem, we need the following definition: we say that $k$ is good with respect to  $\mathcal{H}(W,S,\phi)$  if
   $f_{\blambda}1_k\neq 0$ for all $\blambda\in \Lambda$.  
The proof of the existence of basic set for Hecke algebras (in characretistic $0$) with respect to the $a$-function 
 has been given Geck and  Geck-Rouquier (see \cite{Gsurvey}) when$ \Phi$ is positive. 
  The following proposition obtained in \cite[Prop 2.5]{CJ} allows the  extension of the result for arbitrary $\Phi$.
   This result will also be crucial in the rest of the paper.
\begin{Prop}\label{cj}
{For all }$\blambda\in \Lambda$, we have:
 $$a^{\phi}( V^{\blambda}_{\phi})=a^{|\phi|} (V^{\blambda^{\varepsilon}}_{|\phi|})$$
 \end{Prop}
The theorem of existence becomes then the following:
\begin{Th}\label{main}
We keep the above notations. Assume  in addition that  Lusztig's conjectures {\bf P1-P15} in \cite[\S 14.2]{lu} hold
 and that $k$ is good with respect to  $\mathcal{H}(W,S,\phi)$. Then $\mathcal{H}(W,S,\phi)$ admits a basic set $\mathcal{B}(\phi)\subset \Lambda$ with respect to
  any specialization and to the map $a^{\phi}$,  the Lusztig $a$-function. This basic set is called the canonical basic set and it only depends on $e$ and $\phi$.
\end{Th}
\begin{Rem}  Lusztig's conjectures are known to hold in the following cases:
\begin{enumerate}
\item For all finite Weyl group, in the so called ``equal parameter case", that is when there exists $a\in \mathbb{N}$ such that $\phi (s)=a$ for all $s\in S$ by Lusztig (see \cite[Ch. 15]{lu}.)
\item In type $B_n$, in the so called ``asymptotic case" by the works of  Bonnaf\'e-Iancu \cite{BI}, Bonnaf\'e \cite{B2}, Geck \cite{Gconj} and Geck-Iancu \cite{GI}.
\end{enumerate}
\end{Rem}

In fact, the results in \cite{CJ} show that  the canonical basic set  $\mathcal{B}(\phi)\subset \Lambda$ can 
 be deduced from      $\mathcal{B}(|\phi|)$.  We illustrate this fact in type $B_n$ which is our main centre of interest in this paper. 
 Hence, let $W$ be a Weyl group of type $B_n$.
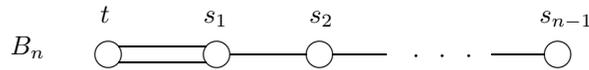
\begin{center}
\begin{picture}(250,30)
\put(  3, 08){$B_n$}
\put( 40, 08){\circle{10}}
\put( 44, 05){\line(1,0){33}}
\put( 44, 11){\line(1,0){33}}
\put( 81, 08){\circle{10}}
\put( 86, 08){\line(1,0){29}}
\put(120, 08){\circle{10}}
\put(125, 08){\line(1,0){20}}
\put(155, 05){$\cdot$}
\put(165, 05){$\cdot$}
\put(175, 05){$\cdot$}
\put(185, 08){\line(1,0){20}}
\put(210, 08){\circle{10}}
\put( 37, 20){$t$}
\put( 76, 20){$s_1$}
\put(116, 20){$s_2$}
\put(203, 20){$s_{n-1}$}
\end{picture}
\end{center}
In this case, $\Lambda$ can be defined to be the set $\Pi_n^2$  of bipartitions of rank $n$.  
 In the following, it will be useful to introduce a ``more generic" Hecke algebra than the one defined
 in the introduction of this section.
 
 %Let us now describe in more details the represnetation theory of these algebras. 
Let $V$ and $v$ be indeterminates and consider the 
generic Hecke algebra  $\mathcal{H}(\{V,v\})$ of type $B_n$ over $\mathbb{Z}[V^{\pm 1/2},v^{\pm 1/2}]$ with presentation as follows:
$$\begin{array}{rl}
(T_0-V)(T_0+1)=0& \\
(T_i-v)(T_i+1)=0& \textrm{if } i=1,...,n-1
\end{array}$$
Let $K$ be the field of fractions of $A$. We set 
$$\operatorname{Irr} (\mathcal{H}_K (\{V,v\}))=\{V^{\blambda}\ |\ \blambda\in \Pi^2_n\}.$$
%As in \S \ref{int}, by Tits' deformation theorem, we have a canonical bijection between $\operatorname{Irr}( \mathcal{H}_K(W,S,\phi))$ 
    and $\Lambda=\Pi^2_n$.

We assume that we have a specialization $\theta:\mathbb{Z}[V^{\pm 1/2},v^{\pm 1/2}]\to k$ where $k$ is a field. By
 results of Dipper and James \cite[Th. 4.17]{DJ}, one can restrict ourselves to the  following case. 
 We  assume that  $\theta (V)=-q^{da}$ and $\theta (v)=q^a$ for $(a,b)\in \mathbb{N}^2$ and  for $q\in k^{\times}$. 
The resulting algebra   $\mathcal{H}_{k}(\{q^a,-q^{ad}\})$ has a presentation as follows:
 $$\begin{array}{rl}
(T_0+q^{da})(T_0+1)=0& \\
(T_i-q^a)(T_i+1)=0& \textrm{if } i=1,...,n-1.
\end{array}$$
It is in general a non semisimple algebra and as in \S \ref{int}, we have an associated decomposition matrix.
$$D=([V^{\blambda}:M])_{\blambda\in \Pi^2_n,M\in \operatorname{Irr}(\mathcal{H}_{k}(\{q^a,-q^{ad}\}))}$$
By a deep theorem of Ariki \cite[Thm 14.49]{Alivre}, this matrix is nothing but the (evaluation at $v=1$ of the) 
 matrix of the canonical basis for the irreducible highest weight $\mathcal{U}_v(\widehat{\mathfrak{sl}_e})$-module 
 with weight $\Lambda_d+\Lambda_{0}$ (where the $\Lambda_i$ with $i\in \mathbb{Z}/e\mathbb{Z}$ denote the fundamental weights).
 
% We end this section by giving the effects of the map  $\Gamma: \mathcal{H}(W,S,\phi)\to \mathcal{H}(W,S,\phi')$
 %on the simple  $\mathcal{H}_K (W,S,\phi)$-modules. Keeping the same notations as  above, we have:
Applying \cite{CJ}  to type $B_n$ leads the following proposition. 
 \begin{Prop}\label{B} Let $(W,S)$ be the Weyl group of type $B_n$. 
 We assume that we have a map $\phi :S=\{t,s_1,...,s_{n}\}\to \mathbb{Z}$  satisfying $(\star)$. We set 
 $|\phi| :S=\{t,s_1,...,s_{n}\}\to \mathbb{Z}$  such that $|\phi|(s)= |\phi(s)|$ for all $s\in S$. Then, for 
 all specialization $\theta$, we have 
 $$\mathcal{B}(\phi) =\left\{ (\lambda^0,\lambda^1)^{\varepsilon}\ |\  (\lambda^0,\lambda^1)\in \mathcal{B}(|\phi|)\right\}$$ 
 where
 \begin{enumerate}
 \item $(\lambda^0,\lambda^1)^{\varepsilon}=(\lambda^1,\lambda^0)$ if $S^+=\{s_1,...,s_{n-1}\}$ and $S^-=\{t\}$,
  \item $(\lambda^0,\lambda^1)^{\varepsilon}=({\lambda^1} ',{\lambda^0 }')$ if $S^+=\emptyset $ and $S^-=\{s_1,...,s_{n-1},t\}$,
   \item $(\lambda^0,\lambda^1)^{\varepsilon}=({\lambda^0} ',{\lambda^1} ')$ if $S^+=\{t\}$ and $S^-=\{s_1,...,s_{n-1}\}$,
 \end{enumerate}
 \end{Prop}
 From now, we will denote $\kappa (\lambda^0,\lambda^1):=(\lambda^1,\lambda^0)$.

 %The next sections will be devoted to a deepest study of this case.  The aim will be to find the basic sets in type $B_n$ for all choices of parameters. 
 % In the following section, we study the algebra  $\mathcal{H}_{k}(\{q^a,-q^{ad}\})$ where $q$ is generic.
% \begin{Rem}
%Regarding Ariki's theorem, one has, again, to use the automorphism $T_0\mapsto -T_0$ to see the
 %correspondence with the Fock space representation.
%\end{Rem}

% For the exceptional types see \cite{GJlivre}. The type $D_n$ will be discussed at the end of the paper. 
 %We now study type $B_n$ which, by the multiple possible choices of the weight function, turn out to be the more complicated  (and also the most interesting) case.

%Finally  \cite[Thm. 3.3]{Jbasic} shows that, assuming the existence of basic sets, it suffices to describe them in charateristic $0$ to obtain the basic sets in all cases.  Hence, 
 %from now, we will only work in  characteristic $0$. 

\subsection{Constructible representations}  
From the above definitions and the results in \cite{B}, it will be easy to extend known results  of constructible representations as defined by
 Lusztig \cite[\S 22.1]{lu}. 
 Let $(W,S)$ be a Weyl group and let  $\phi$ be  a map $\phi:S \to \mathbb{Z}$ satisfying $(\star)$.
Let $I\subset S$ and let $(W_I,I)$ be the corresponding parabolic
subgroup. Let  $\phi_I$ be the restriction of 
 $\phi$ to $I$.  Each simple $\mathbb{C}[W]$-module $E^{\blambda}$  can be seen as a specialization of a simple 
$\mathcal{H}(W,S,\phi)$-module $V^{\blambda}_{\phi}$. 
 By \S \ref{basic},  each simple $\mathbb{C}[W]$-module  
 $U$ comes equipped with  
an associated invariant $a^{\phi}(U)$ depending on the choice of $\phi$.  
Let $U$ be a simple $\mathbb{C}[W_I]$-module. Then we can  uniquely write:
$$\textrm{Ind}_I^S (U)=U_{(0)}\oplus U_{(1)} \oplus ....$$
where for any integer $i$
$$U_{(i)}=\bigoplus_{V} [\textrm{Ind}_I^S (U):V]V\ (\textrm{sum over all $V\in \operatorname{Irr}(\mathbb{C}[W])$ such that $a^{\phi} (V)=i$}).$$
Then the ${\mathbf J}$-induction of a simple  $\mathbb{C}[W_I]$-module $U$ is 
$$\mathbf{J}_I^S (U)=U_{(a^{\phi_I}(U))}   $$
Using this ``truncated" induction, the constructible representations with respect to $\phi$ 
 are defined inductively in the following way:
\begin{enumerate}
\item If $W=\{ 1 \}$, only the trivial module is constructible.
\item If $W \neq \{ 1 \}$, the set of
constructible $\mathbb{C}[W]$-modules consists of the $\mathbb{C}[W]$-modules of the form 
\[
{\mathbf{J}}^S_I  (V) 
\quad{\mbox{or}}\quad 
{\mathrm{sgn}} \otimes {\mathbf{J}}^S_I (V) ,
\]
where ${\mathrm{sgn}}$ is the sign representation of $W$,
 and $I$ some proper subset of $S$.
\end{enumerate}
One can define an analogue of the decomposition matrix for this type of representations: the constructible matrix $D^{\phi}_{\textrm{cons}}$ which is defined as follows.
\begin{itemize}
\item The rows are labelled by $\Lambda$,
\item the coefficients in a fixed column give the expansion of the corresponding constructible representation in terms of the irreducible ones.
\end{itemize}

\begin{Prop}\label{changecons}
Let $\phi:S\to \mathbb{Z}$ be a map satisfying $(\star)$. 
Keeping the above notations $\bigoplus_{\blambda\in \Lambda} \alpha_{\blambda} V_{\phi}^{\blambda}$ is a constructible $\mathbb{C}[W]$-module  with respect to $\phi$ 
 if and only if  $\bigoplus_{\blambda\in \Lambda} \alpha_{\blambda} V^{\blambda^{\varepsilon}}_{|\phi|}$  is a constructible $\mathbb{C}[W]$-module  with respect to $|\phi| $ 
\end{Prop}
\begin{proof}
The result  follows directly from Prop. \ref{cj} and the definition of the constructible representation. %(Goldschmidt 1980, (3.4)].)).
\end{proof}

\begin{Rem}
If $\phi$ is positive, the blocks of the matrix $D^{\phi}_{\textrm{cons}}$ are known as the families of characters. 
 This notion plays an important role in the theory of reductive group. In fact, it is possible 
  to generalize the definition  of  families of characters to any map $\phi$ (ie. not necessary positive) and to  a wider class of algebras : the cyclotomic Hecke algebras, using the notion
   of Rouquier blocks. 
   The associated families have been studied by Chlouveraki \cite{Chlou}. When $W$ is a Weyl group,
      one can easily see that for any map $\phi$, 
      the blocks of the matrix $D^{\phi}_{\textrm{cons}}$ corresponds to the families of characters as given in \cite{Chlou}. 
\end{Rem}
\section{Bijections of basic sets}

 We now focus on the case of the Hecke algebras of type $B_n$.  
Before the study of the representation theory, we first give the formula for the Lusztig $a$-function in the case where 
 $\phi (t)=b\geq 0$ and $\phi (s_i)=a\geq 0$ for $i=1,...,n-1$.
\subsection{$a$-function in type $B_n$}\label{aBn}  We first introduce a combinatorial object which will be
 useful in the following: the shifted symbol of a bipartition.  
 Let $\beta=(\beta_1,...,\beta_k)$ be a sequence of strictly increasing integers and let $s$ be a rational nonnegative number. 
 We denote by $[s]$ the integer part of $s$. 
 We set
 $$\beta (s):=(s-[s],s-[s]+1,...,s-1,\beta_1+s,...,\beta_k+s).$$
Let $r\in \mathbb{Q}$ and let $\blambda=(\lambda^0,\lambda^1)$ be a bipartition of
 rank $n$. Let $h^{0}$ and $h^1$ be the heights of the
 partitions $\lambda^{0}$ and $\lambda^1$ and let $h$ be a positive integer such that
  $h\geq \textrm{max}(h^0,h^1)+1$. We say that $h$ is an admissible size for $\blambda$. 
We define  two sequences of strictly decreasing integers
$$\beta^{0}=(\lambda^{0}_{h}-h+h,...,\lambda^{0}_{j}-j+h,   ...,\lambda^{0}_1-1+h)$$
and 
$$\beta^{1}=(\lambda^{1}_{h}-h+h,...,\lambda^{1}_j-j+h,   ...,\lambda^{1}_1-1+h)$$
The {\it shifted $r$-symbol} of $\blambda$ of size $h$ is  then the family of sequence 
$${\bf B}_r (\blambda):=(B^{0},B^{1})$$
such that %for $j=0,1$, we have 
$$B^{0}=\left\{ \begin{array}{ll}
\beta^0 (r)& \textrm{if }r\geq 0\\ 
\beta^1 (-r) & \textrm{otherwise}.\end{array}\right. \qquad \textrm{and}\qquad 
B^{1}=\left\{  \begin{array}{ll}
\beta^0 & \textrm{if }r\leq 0\\ 
\beta^1  & \textrm{otherwise}.\end{array}\right.$$
The shifted symbol  of size $h$ is usually written as a two row tableaux as follows:
 $${\bf B}_r(\blambda):=\left(\begin{array}{c}
 B^{0}\\
 B^{1}
 \end{array}\right)$$
 Note that ${\bf B}_r (\blambda)={\bf B}_{-r} (\kappa(\blambda))$.
 
\begin{Exa}
Let $r=1/2$ and $\blambda=(2.1,3.3.2)$. We have $h^0=2$ and $h^1=3$. Then the associated shifted $r$-symbol of size $4$ is
 $$\left(\begin{array}{cccc}
 1/2& 3/2  &7/2& 11/2\\
 0&3&5&6
 \end{array}\right)$$
Let $r=-5/2$ and $\blambda=(2.1,3.3.2)$. We have $h^0=2$ and $h^1=3$ 
 then the associated shifted $r$-symbol of size $4$ is
 $$\left(\begin{array}{cccccc}
  1/2&3/2&  1/2&7/2&11/2&13/2\\
 0& 1  &3& 5& 
 \end{array}\right)$$
\end{Exa}

  Assume now that $\phi (t)=b\geq 0$ and $\phi (s_i)=a> 0$ for $i=1,...,n-1$. Let $\blambda\in \Lambda$
  and  let ${\bf B}_{b/a} (\blambda)$ be the shifted $b/a$ symbol of $\blambda$ of size $h$.  Let 
  $$\gamma_1\geq \gamma_2\geq ....\geq \gamma_{t},$$
be the elements of this symbol written in decreasing order (with repetition). Then write
$$a_1^{\phi,h} (\blambda)=\sum_{i=1}^t (i-1)\gamma_i.$$
Then by Lusztig \cite[\S 22.14]{lu}, we have:
$$a^{\phi}(\blambda)=a_1^{\phi,h} (\blambda)-a^{\phi,h}_1 (\bemptyset).$$
Regarding the above definition and the connection between Lusztig $a$-function when $b<0$ and $b>0$ given by Prop. \ref{cj}, we deduce:
\begin{Prop}
Assume that $\phi (t)=b$ and $\phi (s_i)=a\geq 0$ for $i=1,...,n-1$. Let $\blambda\in \Lambda$
  and  let ${\bf B}_{b/a} (\blambda)=(B^{0},B^{1})$ be the shifted $b/a$-symbol of $\blambda$.  Let 
  $$\gamma_1\geq \gamma_2\geq ....\geq \gamma_{t}$$
be the elements of this symbol written in decreasing order (with repetition).  We denote $\gamma_{b/a}^h (\blambda)=(\gamma_1,...,\gamma_t)$ the associated partition.  
We then write
$$a^{\phi}_1 (\blambda)=\sum_{i=1}^t (i-1)\gamma_i$$
Then  we have:
$$a^{b/a}(\blambda):=a^{\phi}(\blambda)=a^{\phi,h}_1 (\blambda)-a^{\phi,h}_1 (\bemptyset)$$
It does not depend on the size of the symbol.
\end{Prop}
\begin{proof}
When $b\geq 0$, this is a result of Lusztig. Assume that $b<0$, then we have by Prop. \ref{cj}, 
$$a^{b/a}(\blambda):=a^{-b/a}(\kappa(\blambda))$$
and ${\bf B}_{b/a} (\blambda)={\bf B}_{-b/a} (\kappa(\blambda))$ so the result follows 
\end{proof} 

The following result is a direct consequence of the formula of the $a$-function.  
\begin{Cor}\label{orde}
Assume that $\phi (t)=b$ and $\phi (s_i)=a\geq 0$ for $i=1,...,n-1$.  
Let $\blambda\in \Lambda$ and $\bmu\in \Lambda$. Let $h$ be an admissible size for 
$\blambda$ and $\bmu$. Assume that  $\gamma^h_{b/a} (\blambda)\rhd \gamma^h_{b/a} (\bmu)$ then we have
 $a^{b/a} (\blambda)<a^{b/a} (\bmu)$.

\end{Cor}

\subsection{Constructible representations in type $B_n$}

 Let $V$ and $v$ be indeterminates and consider the Hecke algebra  $\mathcal{H}(\{V,v\})$ of type $B_n$ over $\mathbb{Z}[V^{\pm 1/2},v^{\pm 1/2}]$ as in \S \ref{basic}. We have
$$\operatorname{Irr} (\mathcal{H}_K (\{V,v\}))=\{V^{\blambda}\ |\ \blambda\in \Pi^2_n\}.$$
We assume that we have a specialization $\theta:\mathbb{Z}[V^{\pm 1/2},v^{\pm 1/2}]\to \mathbb{Q}(q^{\frac{1}{2}})$ where $q$ is an indeterminate 
 such that $\theta (V)=-q^{da}$ and $\theta (v)=q^a$ for $(a,b)\in \mathbb{N}^2$.  For $d\in \mathbb{Z}$. 
Then we have  an associated decomposition matrix. 
$$D_{\theta}=([V^{\blambda}:M])_{\blambda\in \Pi^2_n,M\in \operatorname{Irr}(\mathcal{H}_{\mathbb{Q}(q^{1/2})}(q^a,-q^{da}))}$$
%First assume that $d\geq 0$.
 By Ariki's theorem,  this decomposition matrix
  is the matrix of the canonical basis of the irreducible highest weight  $\mathcal{U}_v ({\mathfrak{sl}_{\infty}})$-module with weight 
  $\Lambda_d+\Lambda_0$. Assume that $d\geq 0$ then by \cite{LM} it has another interpretation in terms of Kazhdan-Lusztig theory: this is the matrix $D_{\textrm{cons}}^{\phi}$ 
   of the constructible representations for the algebra $\mathcal{H}(W,S,\phi)$ where $\phi (t)=d$ and $\phi (s_i)=1$ for all $i=1,...,n-1$. Thus, we have 
   $D_{\textrm{cons}}^{\phi}=D_{\theta}$. 
   In other words,
    the columns of this decomposition  matrix give the expansion of the constructible representations  associated to the map $\phi$ in terms of 
     the simple $\mathcal{H}(W,S,\phi)$-modules $V^{\bmu}_{\phi}$.
     
      It is natural to ask if a basic set as defined in Def. \ref{basique}  can be found in this situation that is if one can order the rows and columns of the constructible matrix such that it has a unitriangular shape.
       The explicit determination of  the  matrix $D_{\theta}$ has been given by Lusztig and by Leclerc-Miyachi using different technics when $d\geq 0$.
      We here follows the latter exposition \cite{LM} and extend it to the case $d\in \mathbb{Z}$.
      
       Hence we assume from now that $\phi$ is such that    $\phi (t)=d\in \mathbb{Z}$ and $\phi (s_i)=1$ for all $i=1,...,n-1$.  First, we need 
    some additional combinatorial definition.  Let $\blambda$ be a bipartition and consider its $d$-symbol ${\bf B}_d (\blambda)=(B^0,B^1)$. We say 
     that $\blambda$ is standard, or equivalently that ${\bf B}_d (\blambda)$ is standard if
   $$B^{1}_i\geq B^{0}_i\ \textrm{for all $i\geq 1$}.$$
   The set of standard bipartitions is denoted by $\textrm{Std}(d)$.

      Let ${\bf B}_d (\blambda)$ be a standard symbol. We define an injection $\Psi:B^{1}\to B^{0}$ such that $\Psi (j)\leq j$ for all 
       $j\in B^{1}$. It is obtained by describing the subsets 
       $$B^{1}_l:=\{ j\in B^{1}\ |\ \Psi (j)=j-l\}.$$
       We set $B^{1}_0=B^{1}\cap B^{0}$ and for $l\geq 1$, we put:
       $$B^{1}_l=\{j\in B^{1}\setminus \{ B^{1}_0,...,B^{1}_{l-1}\}\ |\ j-l\in B^{0}\setminus \Psi (B^{1}_0\cup ...\cup B^{1}_{l-1})\}$$
       the pairs $(j,\Psi (j))$ with $\Psi (j)\neq j$ are called the pairs of the symbols ${\bf B}_d(\blambda)$. Let $\mathcal{C}(\blambda)$ be the
        set of all bipartitions $\bmu$ such that the symbol of $\bmu$ is obtained from ${\bf B}_d (\blambda)$ by permuting some pairs in 
       ${\bf B}_d (\blambda)$ and reordering the rows.  
       We also define $\operatorname{Inv}_{d}(\blambda)$ to be the bipartition in $\mathcal{C}(\blambda)$ whose symbols is obtained from the
        symbol of $\blambda$  
        after all possible permutations of the pairs. 
     
 The following is a result by  Lusztig and  Leclerc-Miyachi when $d\geq 0$ which is easily extend 
  in the case where $d\in \mathbb{Z}$ by Prop. \ref{changecons} and by the definition of symbols.
      \begin{Prop}\label{cons}
      Assume that $\phi$ is such that    $\phi (t)=d\in \mathbb{Z}$ and $\phi (s_i)=1$ for all $i=1,...,n-1$.  
     The constructible representations with respect to $\phi$ are labelled by the standard bipartitions. Moreover, if $\blambda$ is 
      a standard bipartition, the associated constructible representation  with respect to $\phi$ is  
      $$\bigoplus_{\bmu\in \mathcal{C}(\blambda)} V_{\phi}^{\bmu}.$$
      \end{Prop}

\begin{Exa}
We consider the Weyl group of type $B_3$. Let $\phi$ be such that    $\phi (t)=2\in \mathbb{Z}$ and $\phi (s_i)=1$ for all $i=1,2$.  
 From above, one can check that the only non trivial constructible representations are:
 $$V_{\phi}^{(\emptyset,3)}\oplus V_{\phi}^{(1,2)},\ V_{\phi}^{(1,2)}\oplus V_{\phi}^{(1.1,1)}, V_{\phi}^{(1.1,1)}\oplus V_{\phi}^{(1.1.1,\emptyset)}$$
In other words, we have:
$$D^{\phi}_{\textrm{cons}}=\left(
\begin{array}{ccccccccc}
   1&0&0&0&0&0&0&0&0\\
  0&1&0&0&0&0&0&0&0\\
 0&0&1&0&0&0&0&0&0\\
   1&0&0&1&0&0&0&0&0\\
  0&0&0&0&1&0&0&0&0\\
    0&0&0&0&0&1&0&0&0\\
 0&0&0&1&0&0&1&0&0\\
    0&0&0&0&0&0&0&1&0\\
 0&0&0&0&0&0&0&0&1\\
   0&0&0&0&0&0&1&0&0
  \end{array}\right)
  \begin{array}{ccccccccc}
   (\emptyset,3)\\
   (\emptyset,2.1)\\   
  (\emptyset,1.1.1)\\
 (1,2)\\
 (1,1.1)\\
  (2,1)\\
(1.1,1)\\
(3,\emptyset)\\
(2.1,\emptyset)\\
(1.1.1,\emptyset)
  \end{array}$$
Note that the constructible representations are labelled by the set of standard bipartitions which is 
in this case $\Pi^2_3\setminus \{(1.1.1,\emptyset )\}$. 
Note also that the unique non trivial block of the above matrix is given by the following set of bipartitions 
$$\{{(\emptyset,3)},{(1,2)}, {(1.1,2)}, {(1.1.1,\emptyset)}\}$$
which corresponds to the unique non trivial family of characters.

Now, if we set $\phi'$ be such that    $\phi' (t)=-2\in \mathbb{Z}$ and $\phi' (s_i)=1$ for all $i=1,2$.  
 From above, one can check that the only non trivial constructible representations are:
 $$V_{\phi'}^{(3,\emptyset)}\oplus V_{\phi'}^{(2,1)},\ V_{\phi'}^{(2,1)}\oplus V_{\phi'}^{(1,1.1)}, V_{\phi'}^{(1,1.1)}\oplus V_{\phi'}^{(\emptyset,1.1.1)}$$
The unique non trivial block of the above matrix is given by the following set of bipartitions 
$$\{{(3,\emptyset)},{(2,1)}, {(1,1.1)}, {(\emptyset,1.1.1)}\}$$
which corresponds to the unique non trivial family of characters.
\end{Exa}

\subsection{Basic sets of constructible representations}
The aim of this section is to show the existence of basic sets for the matrix affording the constructible representations. 
 To do this, we have to study the matrix 
$D_{\textrm{cons}}^{\phi}$ 
   of the constructible representations for the algebra $\mathcal{H}(W,S,\phi)$ where $\phi (t)=d\in \mathbb{Z}$ and $\phi (s_i)=1$ for all $i=1,...,n-1$. 
    We denote by $\textrm{Cons} (d)$  the set of constructible $\mathcal{H}(W,S,\phi)$-modules. 
 The main result of this section is the following one. 
\begin{Th}\label{conB}
We keep the above notations and we put $r\in \mathbb{Q}$ such that $r\neq d$.
\begin{enumerate}
\item   For all $U\in \operatorname{Cons} (d) $ there exists $\lambda_U\in \Lambda$ such that
 $$[V^{\blambda_U},U] =1\textrm{ and }a^{r}(\bmu)>a^{r}(\blambda_U)\textrm{ if }[V^{\bmu},U] \neq 0$$
\item Let 
 $$\mathcal{B}_{\infty}^r:=\{\lambda_U\ |\  U\in \operatorname{Cons} (d)  \}$$
Then the map 
 $$\begin{array}{cll}
\operatorname{Cons}(d)   &\to& \mathcal{B}_{\infty}^r\\
U &\mapsto& \lambda_U
\end{array}$$
is a bijection
\item We have 
 $$\mathcal{B}_{\infty}^r=\left\{\begin{array}{ll}
 \operatorname{Std}(d)& \textrm{if } r<d\\
 \operatorname{Inv}_{d}( \operatorname{Std}(d))& \textrm{if } r>d\end{array}\right.$$
 \end{enumerate}

\end{Th}
\begin{proof}
Let us first assume that  $d\geq 0$.  
Let  $\blambda$ be a standard bipartition then by Prop \ref{cons}, we have an associated constructible representation labelled by this bipartition.
Assume that $V^{\bmu}_{\phi}$ appears as an irreducible constituent   
   in the expansion of this  constructible representation. Then $\bmu$ can be constructed from $\blambda$   by permuting some pairs in the shifted $d$-symbol of $\blambda$. Let
    $h$ be an admissible size for  $\blambda$. By construction, it implies that this size is admissible for $\bmu$. Let  
${\bf B}_d (\bmu)=(B^{0},B^{1})$ be the shifted $d$-symbol of $\bmu$  of size $h$. By \S \ref{aBn}, we have
 $B^0=\beta^0 (d)$ and $a=\beta^1$. 
\begin{itemize}
\item If $r>d$  the shifted $r$-symbol 
  of $\bmu$ is ${\bf B}_r (\bmu)=(B^{0}(r-d),B^{1})$. 
   By the definition of $\mathcal{C}(\blambda)$,
    It is easy to see that 
   $$\gamma^h_{r}(\operatorname{Inv}_{d}(\blambda))\unrhd  \gamma^h_{r}(\bmu),$$
for all $\bmu\in \mathcal{C}(\blambda)$.   
 Hence by Cor  \ref{orde},  we obtain:
$$   a^r (\operatorname{Inv}_{d}(\blambda))\leq a^r (\bmu)   $$
\item if $r<d$, the shifted $r$-symbol of $\bmu$ of size $h$ is ${\bf B}_r (\bmu)=(B^{0}(d-r),B^{1})$ if  $r$ is positive.
 In this case, by the definition of $\mathcal{C}(\blambda)$ It is easy to see that 
   $$\gamma_r^h(\blambda)\unrhd  \gamma_r^h(\bmu),$$
for all $\bmu\in \mathcal{C}(\blambda)$  such that $\bmu\neq \blambda$ . Hence we obtain:
$$   a^r (\blambda)\leq a^r (\bmu)   .$$
If $r$ is negative, the shifted $r$-symbol of $\bmu$ of size $h+d$ is ${\bf B}_r (\bmu)=(B^{1}(d-r),B^{0})$.
 In this case, by the definition of $\mathcal{C}(\blambda)$ It is easy to see that 
   $$\gamma_r^{h+d}(\blambda)\rhd  \gamma_r^{h+d}(\bmu),$$
and we can conclude as above. \end{itemize}
To summarize, We know that an arbitrary column of the constructible matrix $D_{\theta}$ is naturally labelled
 by a standard bipartition $\blambda_U$ with $U\in \textrm{Cons} (d)$. The 
 above discussion shows that the minimal bipartition $\blambda$  with respect to $a^r$ such that 
$[V^{\blambda}:M]\neq 0$ is $\blambda=\blambda_U$ if $r<d$, $\blambda=\textrm{Inv}_d (\blambda_U )$ otherwise. In addition
 by Prop \ref{cons}, if $[V^{\blambda}:U]\neq 0$  then $[V^{\blambda}:U]= 1$.  This proves the Theorem if $d>0$.
 
The case $d<0$ is deduced from the above one using Prop. \ref{cj} and the facts that $\kappa(\textrm{Std} (d))=\textrm{Std} (-d)$,  $\kappa(\textrm{Inv}_d (\textrm{Std} (d)))=\textrm{Inv}_{-d}(\textrm{Std} (-d))$

\end{proof}

 In the expansion of a constructible character, all
 the simple modules have the same value with respect $a^d$. This follows from 
 the definition of constructible representations and can also be easily seen in the formula above. Theorem \ref{conB}
  shows that  modifying the $a$-function by adding an integer $s$ to $d$ leads  to a natural way to order the rows and columns of the constructible matrix 
  so that it is unitriangular. This induces the existence of a canonical basic set which only depends on the sign of $s$.
 
 Let us now describe the consequences on the parametrisations of the simple modules for Hecke algebras 
  of type $B_n$ in the modular case.

%    Each of these ways give 
    % a natural labelling of the constructible representations by a subset of bipartitions. 
  %The orders on the rows and coluns are defined with respect to the $a$-function $a^r$ with $r<d$ or $r>d$. 

\section{Basic sets  in type $B_n$ }
\subsection{Explicit determination in a special case}
Recall that $A:=\mathbb{Z}[q^{1/2},q^{-1/2}]$. Let  $\phi$  such that $\phi (t)=b\geq 0$ and $\phi (s_i)=a>0$ for $i=1,...,n-1$, Let 
$$\theta: A\to \mathbb{Q}(q_0^{1/2})$$
be a specialization such that $\theta (q)=q_0\in \mathbb{C}^{*}$. Let $e\geq 2$ be the multiplicative order of $\eta_e:=q_0^a$.
Assume that 
$q_0^b=-q_0^{ad}$ 
for some $d \in \mathbb{Z}$. We have an associated  decomposition matrix $D_{\theta}$ and a canonical basic set $\mathcal{B}(\phi)$ 
 which  will, from now,  be rather denoted  by $\mathcal{B}_e^{b/a}$.  
 Consider now  $\phi^1$  such that $\phi^1 (t)=b+ae$ and $\phi^1(s_i)=a$ for $i=1,...,n-1$. Applying the specialization $\theta$, we obtain
again the algebra $\mathcal{H}_{k}(q_0^a,-q_0^{da})$.  
Hence we have another basic set denoted by $\mathcal{B}_e^{b/a+e}$ 
which  has the same cardinality as $\mathcal{B}_e^{b/a}$ but is computed with respect 
 to the $a$-function $a^{b/a+e}$. Continuing in this way we obtain several 
  basic sets 
  $$\mathcal{B}^{b/a}, \mathcal{B}_e^{b/a+e},...,\mathcal{B}_e^{b/a+te}, ...$$
  and one can assume that $0\leq b/a<e$.

 These basic sets have been computed in \cite{GJBn} without the assumptions of {\bf P1-P15} in characteristic $0$.  It has been shown that
  the bipartitions labelling these sets are given by the so called Uglov 
   bipartitions. These bipartitions appear as natural labelling of the Kashiwara's crystal basis for
    irreducible highest weight-modules of level two. As several combinatorial definitions are necessary to introduce them, we have chosen to omit the definition of these bipartitions here. We refer to \cite{GJlivre}
     or \cite{J} for details on them.

 \begin{Th}[Geck-Jacon]\label{mainB} We keep the above notations. Let $e\geq 2$ be the multiplicative order of $q^a$
and let $p_0\in \mathbb{Z}$ be such that
\[ d+p_0e<\frac{b}{a}< d+(p_0+1) e.\]
(Note that the above conditions imply that $b/a \not\equiv d \bmod e$.) 
Then for all $t\geq 0$, we have
\[ \mathcal{B}_e^{b/a+te}=\Phi_{e,n}^{(d+(p_0+t)e,0)}\]
where $\Phi_{e,n}^{(d+(p_0+t)e,0)}$ is defined
in \cite[Def. 4.4]{GJBn}. 
\end{Th}
 Without loss of generality, one can (and do) assume that $p_0=0$.  
 By the results in the second section,  the datum of 
  $\phi'$  such that $\phi' (t)=\varepsilon_1 b$ and $\phi' (s_i)=\varepsilon_2 a$ (for $i=1,...,n-1$) 
 with $(\varepsilon_{1},\varepsilon_2)\in \{\pm 1\}^2$ also  yields the existence 
  of a basic set $\mathcal{B}(\phi')$. By Prop. \ref{B}, they can be easily computed 
   using the above theorem.  In particular, the case $\varepsilon_2=1$ and $\varepsilon_1=-1$
    implies the existence of a basic set $\mathcal{B}^{-b/a}$. In fact, considering all the basic sets
    we have already obtained: 
         $$\mathcal{B}^{b/a}, \mathcal{B}_e^{b/a+e},...,\mathcal{B}_e^{b/a+te}, ...$$
    we obtain several other basic sets 
     $$\mathcal{B}^{-b/a}, \mathcal{B}_e^{-b/a-e},...,\mathcal{B}_e^{-b/a-te}, ...$$
  Now recall the  specialization
  $$\theta: A\to \mathbb{Q}(q_0^{1/2})$$
 such that $\theta (q)=q_0$. Looking at the Hecke algebra
  $\mathcal{H}(\{q^{-b},q^a\})$ and applying the specisalization $\theta$, we obtain  a decomposition matrix
   which can be identify with $D^{\theta}$ by \cite[\S 3.1]{CJ} .   Keeping the above notations, we have:
$q_0^{-b+ae}=-q_0^{-ad}$. 
Note that $-b+ae$ and $a$ are both positives. We then obtain
 a basic set $\mathcal{B}^{-b/a+e}$. Actually, using the same
  arguments as above, one obtain several basic sets:
      $$\mathcal{B}^{-b/a+e}, \mathcal{B}_e^{-b/a+2e},...,\mathcal{B}_e^{-b/a+te}, ...$$ 
      Keeping the notation of thm \ref{mainB} (recall that $p_0=0$), note that we have 
      \[ -d<-\frac{b}{a}+e< -d+e.\]
The above theorem gives also the explicit determination of these basic sets:
 \begin{Cor} Keeping the above notations, for all $t> 0$, we have
\[ \mathcal{B}_e^{-b/a+te}=\Phi_{e,n}^{(-d+(t-1)e,0)}\]
where $\Phi_{e,n}^{(-d+(t-1)e,0)}$ is defined in \cite[Def. 4.4]{GJBn}. 
\end{Cor}
Finally, as above, the existence of basic sets 
     $$\mathcal{B}^{-b/a+e}, \mathcal{B}_e^{-b/a+2e},...,\mathcal{B}_e^{-b/a+te}, ...$$
yields the existence of basic sets:
     $$\mathcal{B}^{b/a-e}, \mathcal{B}_e^{b/a-2e},...,\mathcal{B}_e^{b/a-te}, ...$$
by Prop \ref{conB}.

 \subsection{An Action of the affine Weyl group $\widehat{\mathfrak{S}}_2$}
 Let us summarize the different basic sets  we have obtained  and the  parametrizations by
  the Uglov $l$-partitions. The results in \cite{J} allow to change the parametrization of the sets. 
  First by \cite[Prop. 3.1 (2)]{J} For all $(s_0,s_1)\in \mathbb{Z}^2$, we have 
  $$\kappa(\Phi_{e,n}^{(s_0,s_1)})=\Phi_{e,n}^{(s_1,s_0+e)}$$
  Moreover, by \cite[Prop 3.1 (1)]{J},  for all $m\in \mathbb{Z}$, we have:
  $$\Phi_{e,n}^{(s_0,s_1)}=\Phi_{e,n}^{(s_0+m,s_1+m)}$$
To summarize, the following tabular gives the basic sets and the parametrizations
  by the Uglov bipartitions. 

  \begin{center}
\begin{tabular}{|l|l |  |l|l|}
\hline
\hline
  Basic set with & Associated set &    Basic set with & Associated set   \\
    positive parameters &  of bipartitions  & negative parameters &of bipartitions  \\
   \hline
...  &...& ... &...\\
\hline
 $\mathcal{B}_e^{-b/a+te}$& $\Phi_{e,n}^{(0,d -(t-1)e)} $ &$\mathcal{B}_e^{b/a-te}$&$\Phi_{e,n}^{(d-te,0)} $  \\
    \hline
...  &... &... &...\\
\hline
 $\mathcal{B}_e^{-b/a+2e}$&$\Phi_{e,n}^{(0,d-e)} $ & $\mathcal{B}_e^{b/a-2e}$&$\Phi_{e,n}^{(d-2e,0)} $  \\
    \hline
    $\mathcal{B}_e^{-b/a+e}$ & $\Phi_{e,n}^{(0,d )} $ &$\mathcal{B}_e^{b/a-e}$& $\Phi_{e,n}^{(d-e,0)} $ \\ 
\hline
 $\mathcal{B}_e^{b/a}$ & $\Phi_{e,n}^{(d,0)} $ & $\mathcal{B}_e^{-b/a}$  & $ \Phi_{e,n}^{(0,d+e)}   $ \\
 \hline
 $\mathcal{B}_e^{b/a+e}$ & $\Phi_{e,n}^{(d+e,0)} $&$\mathcal{B}_e^{-b/a-e}$&  $\Phi_{e,n}^{(0,d+ 2e)}    $\\
\hline
 $\mathcal{B}_e^{b/a+2e}$ &$\Phi_{e,n}^{(d+2e,0)}$ &$\mathcal{B}_e^{-b/a-2e}$&  $ \Phi_{e,n}^{(0,d+ 3e)}   $\\
\hline
...  &... &... &...   \\
\hline
 $\mathcal{B}_e^{b/a+te}$&$\Phi_{e,n}^{(d+te,0)}$  &$\mathcal{B}_e^{-b/a-te}$&$\Phi_{e,n}^{(0,d+(t+1)e)}   $   \\
\hline
...  & ...&... &...\\
\hline
\end{tabular}
\end{center}
Following \cite{JL},  we set 
$$\mathcal{F}=\left\{ \pm b/a+te\ |\ t\in \mathbb{Z}\right\}$$
Let  $\widehat{\mathfrak{S}}_{2 }$  be the extended affine symmetric group with generators  $\sigma$ and  $%
y_{0}$, $y_{1}$ and relations
\begin{equation*}
y_{0}y_{1}=y_{1}y_{0}, \quad \sigma^2=1, \quad y_{0}=\sigma y_{1}\sigma.
\end{equation*}

One can define an action of $\widehat{\mathfrak{S}}_2$ on $\mathcal{F}$ determined by the following identities. For all $t\in \mathbb{Z}$, we set:
 $$\sigma. (b/a+te)=-b/a-te,\quad y_0. ( b/a+te)=b/a+(t+1)e,$$
 $$ y_0. ( -b/a+te)=-b/a+(t+1)e  $$
One can easily check that this action is well defined. The above dicussion  yields the existence of  an action on the basic sets by setting for all $w\in \widehat{\mathfrak{S}}_2$ and $\gamma\in \mathcal{F}$:
  $$w.\mathcal{B}_e^{\gamma}=\mathcal{B}_e^{w.\gamma}$$
By Prop \ref{conB}, we have:
 $$\begin{array}{ll}
 \sigma.\mathcal{B}_e^{\gamma}&=\kappa (\mathcal{B}_e^{\gamma})\\
 &=\left\{ (\lambda^0,\lambda^1)\ |\ (\lambda^1,\lambda^0)\in \mathcal{B}_e^{\gamma}\right\} \end{array} 
 $$
 Hence, to describe the action of $\widehat{\mathfrak{S}}_2$ on $\mathcal{F}$, it suffices to understand the action of $y_0$ on an arbitrary basic set. 
By the results in \cite{J} (see also the generalizations in \cite{JL}) together with Thm \ref{mainB},  the action of $y_0$ corresponds to a crystal isomorphism. Using 
 the combinatorial study of this isomorphism in this paper,  It 
 can be expressed using the map $\operatorname{Inv}_D$ defined in the previous section. This is given by the following theorem which uses the combinatorics 
  developed in \cite{J}.

  \begin{Th}\label{cont}
  Assume that $\gamma\in \mathcal{F}$. Then there exists $D\in \mathbb{Z}$ such that:
$$\mathcal{B}_e^{\gamma}  =\Phi_{e,n}^{(D-e,0)}$$
 To describe the action of $y_0$, one can assume that $D\geq 0$, then we have
  $$y_0 . \mathcal{B}_e^{\gamma}= \operatorname{Inv}_{D}(  \mathcal{B}_e^{\gamma})$$
  \end{Th}
  \begin{proof}
  Let $\gamma\in \mathcal{F}$. By Thm \ref{mainB}, there exists   $D\in \mathbb{Z}$ such that $\mathcal{B}_e^{\gamma}  =\Phi_{e,n}^{(D-e,0)}$.
   Since we know that   $\sigma.\mathcal{B}_e^{\delta}=\kappa (\mathcal{B}_e^{\delta})$ for all $\delta\in \mathcal{F}$,  one can assume that $D\geq 0$. 
  
  Let $(\lambda^0,\lambda^1)\in  \mathcal{B}_e^{\gamma}=\Phi_{e,n}^{(D-e,0)}$. Then by  \cite[Prop. 3.1]{J}, we have 
  $(\lambda^1,\lambda^0)\in\Phi_{e,n}^{(0,D)}$. Using \cite[Prop. 4.1]{J},  we deduce that 
   $(\lambda^0,\lambda^1)\in \textrm{Std}(D)$. Again, by    \cite[Prop. 4.1]{J}, we have $\kappa(\operatorname{Inv}_{D}(  \lambda^0,\lambda^1))\in \Phi_{e,n}^{(0,D+e)}$ which implies 
  that  $\operatorname{Inv}_{D}(  \lambda^0,\lambda^1)\in  \Phi_{e,n}^{(D,0)}$. The map sending $(\lambda^0,\lambda^1)\in \Phi_{e,n}^{(D-e,0)}$
   to  $ \operatorname{Inv}_{D}(  \lambda^0,\lambda^1)\in \Phi_{e,n}^{(D,0)}    $ is a bijection. 
  \end{proof}

  Hence, remarkably, this action does not depend on $e$ but only on $D$  ! This will be developed in the following section.
% On peut toujours supposer D positif, is suffit de passer de l'autre coté si ca marche pas
\begin{Rem}
Assume that $b/a=-b/a+e$ then $2b=ae$ and we have  $q_0^{2b}=1$. Then we have two cases to consider
\begin{itemize}
\item if $q_0^b=1$ then $q_0^{ad}=-1$ implies that $e$ is even and $d=e/2$ which is impossible 
 because then $q_0^{b}=q_0^{ae/2}=-q_0^{ae/2}$.
\item if $q_0^b=-1$ then $q_0^{ad}=1$ implies $d=e$ and $d =0$.
\end{itemize} 
In this case, note that $\Phi_{e,n}^{(d,0)} =\Phi_{e,n}^{(0,d)} $ and then 
 $\mathcal{B}_e^{b/a} =\mathcal{B}_e^{-b/a+e}$. Hence the above result is coherent with this case.

\end{Rem}

\subsection{Factorization of the decomposition map}  The aim 
 of this section is to  give an interpretation of  Prop. \ref{cont} in terms of constructible representations. 
  We here keep the notations of this proposition.

Let $\mathcal{H}(W,S,\{Q,q^{a}\})$ be the generic Hecke algebra with parameters $Q$ and $q^a$ (where $Q$ and $q$ are indeterminates). We consider a first specialization:
$$\theta_q:\mathbb{Z}[Q^{\pm 1/2},q^{\pm 1/2}]\to \mathbb{Q}(q^{\frac{1}{2}})$$
We obtain a well defined decomposition map:
$$d_{\theta_q}:R (\mathcal{H}_{\mathbb{Q}(Q^{1/2},q^{1/2})}(W,S,\{Q,q^{a}\}))\to R(\mathcal{H}_{\mathbb{Q}(q^{1/2})}(W,S,\{-q^{D.a},q^{a}\}) )$$
and an associated decomposition matrix $D_{\theta_q}$. 
We also have a specialization
$$\theta:A\to k$$
We obtain a well defined decomposition map:
$$d_{\theta}:R (\mathcal{H}_{\mathbb{Q}(q^{1/2})}(W,S,\{-q^{D.a},q^{a}\}))\to R(\mathcal{H}_{\mathbb{Q}(q^{1/2}_0)}(W,S,\{-q_0^{D.a},q_0^{a}  \}) )$$
and an associated decomposition matrix $D_{\theta}$.
On the other hand, one can also define a specialization 
$$\theta_1:\mathbb{Z}[Q^{\pm 1/2},q^{\pm 1/2}]\to \mathbb{Q}(q^{1/2}_0)$$
We obtain a well defined decomposition map:
$$d_{\theta_1}:R (\mathcal{H}_{Q^{\pm 1/2},q^{\pm 1/2}}(W,S,\{Q,q^{a}\}))\to R(\mathcal{H}_{\mathbb{Q}(q^{1/2}_0)}(W,S,\{-q_0^{D.a},q_0^{a}  \}) )$$
and an associated decomposition matrix $D_{\theta_1}$. 
\begin{Th}[Geck \cite{Gfact}, \cite{G2fact} Geck-Rouquier \cite{GRfact}]
The following diagram is commutative

\xymatrix @R=1cm @C=0,001cm{ R (\mathcal{H}_{\mathbb{Q}(q^{1/2},Q^{1/2})}(W,S,\{Q,q^{a}\}))  \ar[rr]^{d_{\theta_1}} \ar[rd]^{d_{\theta_q}} && R(\mathcal{H}_{\mathbb{Q}(q^{1/2}_0)}(W,S,\{-q_0^{Da},q_0^{a}\})) \\ & R (\mathcal{H}_{\mathbb{Q}(q^{1/2})}(W,S,\{-q^{Da},q^{a}  \}) ) \ar[ru]^{d_{\theta}}  }

In other words, we have:
$$D_{\theta_1}=D_{\theta_q}D_{\theta}$$

\end{Th}
%One can note that an analogue of this theorem has been recently proved for the ``quantification" of the decomposition matrices given by Ariki's theorem \cite{JL2} (that is for the matrices of the canonical bases for highest weight $\mathcal{U}_v (\widehat{\mathfrak{sl}_e})$-modules or $\mathcal{U}_v ({\mathfrak{sl}_{\infty}})$-modules).

The following gives a first consequence of this result which can be also easily checked using the results of the previous sections 
 and the properties of Uglov biartitions.
    \begin{Cor} We have:
    $$\mathcal{B}_e^{\gamma}\subset \mathcal{B}_{\infty}^{\gamma}   .$$
    \end{Cor}
\begin{proof}
This follows directly from the above theorem.

\end{proof}
Proposition \ref{cont} and the above result show that the bijection between 
$y_0 .  \mathcal{B}_e^{\gamma}$ and $ \mathcal{B}_e^{\gamma}$ is ``controlled" by the matrix of the
 constructible representations through the above factorization
 in the following sense.
 
Note that $\gamma< D<\gamma+e$.  We have a bijection
$$\Psi:\mathcal{B}_{\infty}^{\gamma}\to  \mathcal{B}_{\infty}^{\gamma+e} $$
which is naturally defined using $D_{\theta_q}$.  Consider a constructible character and the expansion
 of it in the standard basis. This is given by a column of $D_{\theta_q}$. Let $\blambda$ be the element appearing 
 in this column with non zero coefficient and with minimal value with respect to $a^{\gamma}$. Then 
$ \blambda\in\mathcal{B}_{\infty}^{\gamma} $. Let $\bmu$ be the element appearing 
 in this column with non zero coefficient and with minimal value with respect to $a^{\gamma+e}$. Then 
$ \bmu\in\mathcal{B}_{\infty}^{\gamma+e} $. We then set $\Psi (\blambda)=\bmu$. 
Combining this with the above result we get 
$$\Psi(\mathcal{B}_{e}^{\gamma})=   \mathcal{B}_e^{\gamma+e}  .$$

%%%%%%%%%%%%%%%%%%%%%%%%%%%%%%%%%%%%%%%%%%%%%%%%%%%%%%%%%%%%%%%%%%%%%%%%%%%%%

\end{document}